\documentclass[12pt]{article}
\usepackage{amssymb}
\usepackage{amsthm}
\usepackage{amsmath}
\usepackage{graphicx}
\usepackage{enumerate}

\DeclareMathOperator{\Max}{Max}
\DeclareMathOperator{\Min}{Min}

\newtheorem{theorem}{Theorem}[section]
\newtheorem{definition}[theorem]{Definition}
\newtheorem{lemma}[theorem]{Lemma}
\newtheorem{proposition}[theorem]{Proposition}
\newtheorem{observation}[theorem]{Observation}
\newtheorem{remark}[theorem]{Remark}
\newtheorem{example}[theorem]{Example}

\title{Tense logic based on finite orthomodular posets}
\author{Ivan~Chajda and Helmut~L\"anger}
\date{}
\begin{document}

\footnotetext{Support of the research of both authors by the Austrian Science Fund (FWF), project I~4579-N, and the Czech Science Foundation (GA\v CR), project 20-09869L, entitled ``The many facets of orthomodularity'', is gratefully acknowledged.}

\maketitle

\begin{abstract}
It is widely accepted that the logic of quantum mechanics is based on orthomodular posets. However, such a logic is not dynamic in the sense that it does not incorporate time dimension. To fill this gap, we introduce certain tense operators on such a logic in an inexact way, but still satisfying requirements asked on tense operators in the classical logic based on Boolean algebras or in various non-classical logics. Our construction of tense operators works perfectly when the orthomodular poset in question is finite. We investigate the behaviour of these tense operators, e.g.\ we show that some of them form a dynamic pair. Moreover, we prove that if the tense operators preserve one of the inexact connectives conjunction or implication as defined by the authors recently in another paper, then they also preserve the other one. Finally, we show how to construct the binary relation of time preference on a given time set provided the tense operators are given, up to equivalence induced by natural quasiorders.
\end{abstract}


 
\section{Introduction}

At the beginning of the twentieth century it was recognized that the logic of quantum mechanics differs essentially from classical propositional calculus based algebraically on Boolean algebras. K.~Husimi \cite H and G.~Birkhoff together with J.~von~Neumann \cite{BV} introduced orthomodular lattices in order to serve as an algebraic base for the logic of quantum mechanics, see \cite{Be}. These lattices incorporate many aspects of this logic with one exception. Namely, within the logic of quantum mechanics the disjunction of two propositions need not exist in the case when these propositions are neither orthogonal nor comparable. This fact motivated a number of researchers to consider orthomodular posets instead of orthomodular lattices within their corresponding investigations, see e.g.\ \cite{PP} and references therein.

A propositional logic, either classical or non-classical, usually does not incorporate the dimension of time. In order to organize this logic as a so-called {\em tense logic} (or {\em time logic} in another terminology, see \cite{FP}, \cite{FGV}, \cite{RU} and \cite V) we usually construct the so-called {\em tense operators} $P$, $F$, $H$ and $G$. Their meaning is as follows:
\begin{align*}
P & \ldots\text{``It has at some time been the case that''}, \\
F & \ldots\text{``It will at some time be the case that''}, \\
H & \ldots\text{``It has always been the case that''}, \\
G & \ldots\text{``It will always be the case that''}.
\end{align*}
As the reader may guess, we need a certain time scale. For this reason a {\em time frame} $(T,R)$ was introduced. By $T\neq\emptyset$ is meant a set of time, either finite or infinite, and $R\subseteq T^2$ is the relation of time preference, i.e.\ for $s,t\in T$ we say that $s\mathrel Rt$ means $s$ is ``before'' $t$, or $t$ is ``after'' $s$. For our purposes in this paper we will consider only so-called {\em serial relations} (see \cite{CP15a}), i.e.\ relations $R$ such that for each $s\in T$ there exist some $r,t\in T$ with $r\mathrel Rs$ and $s\mathrel Rt$. Of course, if $R$ is reflexive then it is serial. Usually, $R$ is considered to be a partial order relation or a quasiorder relation, see e.g.\ \cite{Bu}, \cite{DG} and \cite{FP}.

Another important task in tense logics is to construct for given tense operators a relation $R$ on a given time set $T$ such that the new tense operators constructed by means of this relation $R$ coincide with the given tense operators.

If the logic in question is based on a complete lattice $(L,\vee,\wedge)$ and the time frame $(T,R)$ is given then the tense operators $P,F,H,G\colon L^T\rightarrow L^T$ can be defined as follows:
\begin{align*}
P(q)(s) & :=\bigvee\{q(t)\mid t\mathrel Rs\}, \\
F(q)(s) & :=\bigvee\{q(t)\mid s\mathrel Rt\}, \\
H(q)(s) & :=\bigwedge\{q(t)\mid t\mathrel Rs\}, \\
G(q)(s) & :=\bigwedge\{q(t)\mid s\mathrel Rt\}
\end{align*}
for all $q\in A^T$ and all $s\in T$ (see e.g.\ \cite{BP}, \cite{CP15b}, \cite E, \cite{FP} and \cite V).

The problem arises when our logic is based on a poset $(A,\leq)$ where joins and meets need not exist. In \cite{CP13} such a situation has been solved by embedding of $A^T$ into a complete lattice. Then we can use the aforementioned definitions of tense operators $P$, $F$, $H$ and $G$, but the disadvantage is that the results of these operators need not belong to $A^T$ again. This motivated us to try another approach, see Section~3.

Moreover, if $\odot$ denotes conjunction and $\rightarrow$ implication within the given propositional logic, we usually ask our tense operators to satisfy the following inequalities and equalities:
\begin{enumerate}[(1)]
\item $P(x)\odot H(y)\leq P(x\odot y)$,
\item $H(x)\odot P(y)\leq P(x\odot y)$,
\item $H(x)\odot H(y)=H(x\odot y)$,
\item $F(x)\odot G(y)\leq F(x\odot y)$,
\item $G(x)\odot F(y)\leq F(x\odot y)$,
\item $G(x)\odot G(y)=G(x\odot y)$,
\item $P(x\rightarrow y)\leq H(x)\rightarrow P(y)$,
\item $H(x\rightarrow y)\leq P(x)\rightarrow P(y)$,
\item $H(x\rightarrow y)\leq H(x)\rightarrow H(y)$,
\item $F(x\rightarrow y)\leq G(x)\rightarrow F(y)$,
\item $G(x\rightarrow y)\leq F(x)\rightarrow F(y)$,
\item $G(x\rightarrow y)\leq G(x)\rightarrow G(y)$.
\end{enumerate}
For our considerations we need the connectives $\odot$ and $\rightarrow$ introduced in an orthomodular poset. In general this task is ambiguous, see \cite{CP13}, \cite{CP16}, \cite{FP} and \cite K, but for finite orthomodular posets (in fact for orthomodular posets of finite height) it was solved by the authors in \cite{CL} where these connectives are introduced in such a way that one gets an adjoint pair. This is the reason why we will restrict ourselves to finite orthomodular posets only.

The connectives introduced in \cite{CL} are everywhere defined, but their results need not be elements of the given orthomodular poset, but may be subsets of this poset containing (incomparable) elements of the same (maximal) truth value. This may cause a problem for the composition of tense operators. Namely, we want to show that our couple $(P,G)$ of tense operators forms a {\em dynamic pair} on the given orthomodular poset $\mathbf A=(A,\leq,{}',0,1)$, i.e.\ the following axioms (P1) -- (P3) should hold (see \cite{CP15a}):
\begin{enumerate}[(P1)]
\item $G(1)=1$ and $P(0)=0$,
\item $p\leq q$ implies $G(p)\leq G(q)$ and $P(p)\leq P(q)$,
\item $q\leq(G\circ P)(q)$ and $(P\circ G)(q)\leq q$.
\end{enumerate}
If $G,P\colon A^T\rightarrow(2^A\setminus\{\emptyset\})^T$ for some time frame $(T,R)$ then we must solve the problem how to define the compositions $G\circ P$ and $P\circ G$ in (P3).

\section{Preliminaries}

The following concepts are taken from \cite{Be} and \cite{Bi}.

Consider a poset $(A,\leq)$. If $a,b\in A$ and $\sup(a,b)$ exists then we will denote it by $a\vee b$. If $\inf(a,b)$ exists, it will be denoted by $a\wedge b$.

A unary operation $'$ on $A$ is called an {\em antitone involution} if $a\leq b$ implies $b'\leq a'$ and if $a''=a$ for each $a,b\in A$. If the poset $(A,\leq)$ has a bottom or top element, this element will be denoted by $0$ or $1$, respectively, and we will write $(A,\leq,0,1)$ in order to express the fact that $(A,\leq)$ is {\em bounded}.

A {\em complementation} on a bounded poset $(A,\leq,0,1)$ is a unary operation $'$ on $A$ satisfying $a\vee a'=1$ and $a\wedge a'=0$ for each $a\in A$. For a bounded poset $(A,\leq,0,1)$ with an antitone involution $'$ that is a complementation we will write $(A,\leq,{}',0,1)$.

\begin{definition}\label{def1}
An {\em orthomodular poset} is a bounded poset $(A,\leq,{}',0,1)$ with an antitone involution $'$ that is a complementation satisfying the following two conditions:
\begin{enumerate}[{\rm(i)}]
\item If $x,y\in A$ and $x\leq y'$ then $x\vee y$ is defined,
\item if $x,y\in A$ and $x\leq y$ then $y=x\vee(y\wedge x')$.
\end{enumerate}
Condition {\rm(ii)} is called the {\em orthomodular law}. Two elements $a,b$ of $A$ are called {\em orthogonal to each other} if $a\leq b'$ {\rm(}which is equivalent to $b\leq a'${\rm)}.
\end{definition}

Let us note that due to De Morgan's laws, (ii) of Definition~\ref{def1} is equivalent to the condition
\begin{enumerate}
\item[(ii')] If $x,y\in A$ and $x\leq y$ then $x=y\wedge(x\vee y')$.
\end{enumerate}
Throughout the paper we will consider only finite orthomodular posets. An example of a finite orthomodular poset which is not a lattice is depicted in Figure~1.

\begin{center}
\setlength{\unitlength}{7mm}
\begin{picture}(18,8)
\put(9,1){\circle*{.3}}
\put(1,3){\circle*{.3}}
\put(3,3){\circle*{.3}}
\put(5,3){\circle*{.3}}
\put(7,3){\circle*{.3}}
\put(9,3){\circle*{.3}}
\put(11,3){\circle*{.3}}
\put(13,3){\circle*{.3}}
\put(15,3){\circle*{.3}}
\put(17,3){\circle*{.3}}
\put(1,5){\circle*{.3}}
\put(3,5){\circle*{.3}}
\put(5,5){\circle*{.3}}
\put(7,5){\circle*{.3}}
\put(9,5){\circle*{.3}}
\put(11,5){\circle*{.3}}
\put(13,5){\circle*{.3}}
\put(15,5){\circle*{.3}}
\put(17,5){\circle*{.3}}
\put(9,7){\circle*{.3}}
\put(1,3){\line(0,2)2}
\put(1,3){\line(1,1)2}
\put(1,3){\line(3,1)6}
\put(1,3){\line(4,1)8}
\put(3,3){\line(-1,1)2}
\put(3,3){\line(1,1)2}
\put(3,3){\line(2,1)4}
\put(3,3){\line(4,1)8}
\put(5,3){\line(-1,1)2}
\put(5,3){\line(0,1)2}
\put(5,3){\line(2,1)4}
\put(5,3){\line(3,1)6}
\put(7,3){\line(-3,1)6}
\put(7,3){\line(-2,1)4}
\put(7,3){\line(3,1)6}
\put(7,3){\line(4,1)8}
\put(9,3){\line(-4,1)8}
\put(9,3){\line(-2,1)4}
\put(9,3){\line(2,1)4}
\put(9,3){\line(4,1)8}
\put(11,3){\line(-4,1)8}
\put(11,3){\line(-3,1)6}
\put(11,3){\line(2,1)4}
\put(11,3){\line(3,1)6}
\put(13,3){\line(-3,1)6}
\put(13,3){\line(-2,1)4}
\put(13,3){\line(0,1)2}
\put(13,3){\line(1,1)2}
\put(15,3){\line(-4,1)8}
\put(15,3){\line(-2,1)4}
\put(15,3){\line(-1,1)2}
\put(15,3){\line(1,1)2}
\put(17,3){\line(-4,1)8}
\put(17,3){\line(-3,1)6}
\put(17,3){\line(-1,1)2}
\put(17,3){\line(0,1)2}
\put(9,1){\line(-4,1)8}
\put(9,1){\line(-3,1)6}
\put(9,1){\line(-2,1)4}
\put(9,1){\line(-1,1)2}
\put(9,1){\line(0,1)2}
\put(9,1){\line(1,1)2}
\put(9,1){\line(2,1)4}
\put(9,1){\line(3,1)6}
\put(9,1){\line(4,1)8}
\put(9,7){\line(-4,-1)8}
\put(9,7){\line(-3,-1)6}
\put(9,7){\line(-2,-1)4}
\put(9,7){\line(-1,-1)2}
\put(9,7){\line(0,-1)2}
\put(9,7){\line(1,-1)2}
\put(9,7){\line(2,-1)4}
\put(9,7){\line(3,-1)6}
\put(9,7){\line(4,-1)8}
\put(8.85,.25){$0$}
\put(.85,2.3){$a$}
\put(2.85,2.3){$b$}
\put(4.85,2.3){$c$}
\put(6.85,2.3){$d$}
\put(8.85,2.3){$e$}
\put(10.85,2.3){$f$}
\put(12.85,2.3){$g$}
\put(14.85,2.3){$h$}
\put(16.85,2.3){$i$}
\put(.8,5.4){$i'$}
\put(2.8,5.4){$h'$}
\put(4.8,5.4){$g'$}
\put(6.8,5.4){$f'$}
\put(8.8,5.4){$e'$}
\put(10.8,5.4){$d'$}
\put(12.8,5.4){$c'$}
\put(14.8,5.4){$b'$}
\put(16.8,5.4){$a'$}
\put(8.85,7.4){$1$}
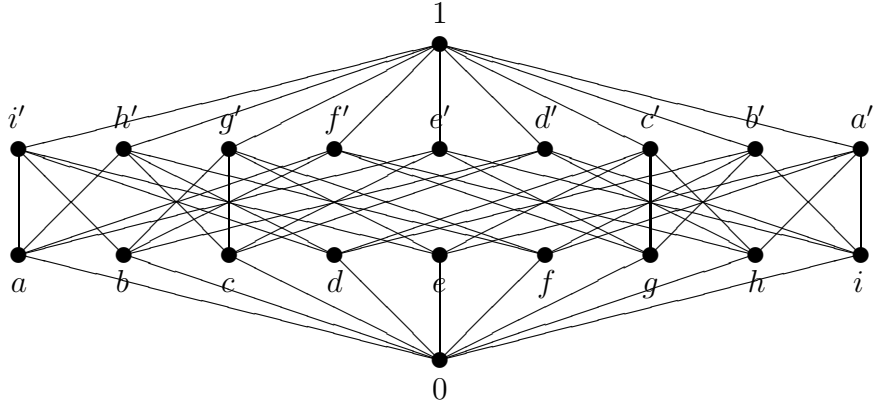
\put(1.8,-.75){{\rm Fig.~1}. Smallest orthomodular poset not being a lattice}
\end{picture}
\end{center}

\vspace*{3mm}

This orthomodular poset is not a lattice since $a\vee b$ does not exist. Namely, $i'$ and $f'$ are minimal upper bounds of $\{a,b\}$. Also $i'\wedge f'$ does not exist since $a$ and $b$ are maximal lower bounds of $\{i',f'\}$.

Consider a poset $(A,\leq)$ and a time frame $(T,R)$ and let $B,C$ be non-empty subsets of $A$ and $x,y\in(2^A\setminus\{\emptyset\})^T$. We define
\begin{align*}
       B\leq C & :\Leftrightarrow b\leq c\text{ for all }b\in B\text{ and all }c\in C, \\
      B\leq_1C & :\Leftrightarrow\text{for every }b\in B\text{ there exists some }c\in C\text{ with }b\leq c, \\
      B\leq_2C & :\Leftrightarrow\text{for every }c\in C\text{ there exists some }b\in B\text{ with }b\leq c, \\
B\sqsubseteq C & :\Leftrightarrow\text{there exists some }b\in B\text{ and some }c\in C\text{ with }b\leq c, \\
       x\leq y & :\Leftrightarrow x(t)\leq y(t)\text{ for all }t\in T, \\
      x\leq_1y & :\Leftrightarrow x(t)\leq_1y(t)\text{ for all }t\in T, \\
      x\leq_2y & :\Leftrightarrow x(t)\leq_2y(t)\text{ for all }t\in T, \\
x\sqsubseteq y & :\Leftrightarrow x(t)\sqsubseteq y(t)\text{ for all }t\in T.
\end{align*}
For $B,C\in2^{(A^T)}\setminus\{\emptyset\}$ we define
\begin{align*}
       B\leq C & :\Leftrightarrow p\leq q\text{ for all }p\in B\text{ and all }q\in B, \\
      B\leq_1C & :\Leftrightarrow p\leq_1q\text{ for all }p\in B\text{ and all }q\in B, \\
      B\leq_2C & :\Leftrightarrow p\leq_2 q\text{ for all }p\in B\text{ and all }q\in B, \\
B\sqsubseteq C & :\Leftrightarrow p\sqsubseteq q\text{ for all }p\in B\text{ and all }q\in B.
\end{align*}
Hereby, for $a\in A$, $a_t\in A$ for all $t\in T$ and $p\in A^T$ we identify
\begin{align*}
& \{a\}\text{ with }a, \\
& \text{the mapping assigning to each }t\in T\text{ the set }\{a_t\}\text{ with the mapping assigning to each} \\
& \quad t\in T\text{ the element }a_t, \\
& \{p\}\text{ with }p.
\end{align*}
Now for $B\subseteq A$ we define
\begin{align*}
L(B) & :=\{x\in A\mid x\leq B\}, \\
U(B) & :=\{x\in A\mid B\leq x\},
\end{align*}
the so-called {\em lower cone} and {\em upper cone} of $B$, respectively. If $(A,\leq)$ is a poset and $B\subseteq A$ then we denote the set of all maximal (minimal) elements of $B$ by $\Max B$ ($\Min B$). If $A$ is finite and $B\neq\emptyset$ then $\Max B,\Min B\neq\emptyset$.

If $a\in A$, $B\subseteq A$ and $a\vee b$ exists for all $b\in B$ then we will denote $\{a\vee b\mid b\in B\}$ by $a\vee B$. Analogously, we proceed if $\vee$ is replaced by $\wedge$.

The following result was proved in \cite{CL}.

\begin{proposition}\label{prop1}
Let $(A,\leq,{}',0,1)$ be an orthomodular poset, $a,b\in A$ and $B,C\subseteq A$. Then the following hold:
\begin{enumerate}[{\rm(i)}]
\item If $a\leq B$ then $B=a\vee(B\wedge a')$,
\item if $C\leq a$ then $C=a\wedge(C\vee a')$,
\item $\Min U(a,b')\wedge b$ and $a'\vee\Max L(a,b)$ are defined.
\end{enumerate}
\end{proposition}

Let us note that both $\leq_1$ and $\leq_2$ are extensions of $\leq$ and are reflexive and transitive, i.e.\ they are so-called {\em quasiorder relations}.

In \cite{CL} we investigated the so-called {\em inexact connectives} in the logic based on a finite orthomodular poset. These are denoted by $\odot$ ({\em conjunction}) and $\rightarrow$ ({\em implication}) and defined by
\begin{enumerate}
\item[(13)] $x\odot y:=\Min U(x,y')\wedge y\quad$ and $\quad x\rightarrow y:=x'\vee\Max L(x,y)$.
\end{enumerate}
Due to Proposition~\ref{prop1} these expressions are everywhere defined and $\odot$ and $\rightarrow$ are binary operators on $A$, more precisely mappings from $A^2$ to $2^A\setminus\{\emptyset\}$. We extend them to $(2^A\setminus\{\emptyset\})^2$ by defining
\begin{align*}
      B\odot C & :=\bigcup\{b\odot c\mid b\in B,c\in C\}, \\
B\rightarrow C & :=\bigcup\{b\rightarrow c\mid b\in B,c\in C\}
\end{align*}
for all $B,C\in2^A\setminus\{\emptyset\}$. The extended operators $\odot$ and $\rightarrow$ are now mappings from $(2^A\setminus\{\emptyset\})^2$ to $2^A\setminus\{\emptyset\}$. In \cite{CL} we showed that $(\odot,\rightarrow)$ form an {\em adjoint pair}, i.e.
\[
x\odot y\sqsubseteq z\text{ if and only if }x\sqsubseteq y\rightarrow z.
\]
 Moreover, we define
\begin{align*}
      (x\odot y)(t) & :=x(t)\odot y(t), \\
(x\rightarrow y)(t) & :=x(t)\rightarrow y(t)
\end{align*}
for all $x,y\in(2^A\setminus\{\emptyset\})^T$ and all $t\in T$. The six-tuple $(A,\leq,\odot,\rightarrow,0,1)$ forms a so-called {\em operator residuated structure}, see \cite{CL} for definition and basic properties. This structure satisfies the so-called {\em divisibility}, it means that
\[
(x\rightarrow y)\odot x\approx\Max L(x,y).
\]
Moreover, $\odot$ satisfies the identities $x\odot x\approx x$ and $x\odot1\approx1\odot x\approx x$. In the following we will investigate tense operators on finite orthomodular posets equipped with the operators $\odot$ and $\rightarrow$ forming a certain propositional logic of quantum mechanics.

\begin{lemma}\label{lem3}
Let $(A,\leq,{}',0,1)$ be a finite orthomodular poset and $p,q\in A^T$. Then
\begin{enumerate}[{\rm(i)}]
\item $p\leq q\rightarrow(p\odot q)$,
\item $(p\rightarrow q)\odot p\leq q$.
\end{enumerate}
\end{lemma}

\begin{proof}
\
\begin{enumerate}[(i)]
\item Since $r\in\Min U(p,q')\wedge q$ implies $r\leq q$ and hence $\Max L(q,r)=r$ and since $q'\leq\Min U(p,q')$ we have
\begin{align*}
q\rightarrow(p\odot q) & =q\rightarrow\big(\Min U(p,q')\wedge q\big)=\bigcup\{q\rightarrow r\mid r\in\Min U(p,q')\wedge q\}= \\
                       & =\bigcup\{q'\vee\Max L(q,r)\mid r\in\Min U(p,q')\wedge q\}= \\
                       & =\bigcup\{q'\vee r\mid r\in\Min U(p,q')\wedge q\}=q'\vee\big(\Min U(p,q')\wedge q\big)= \\
                       & =\Min U(p,q')\geq p
\end{align*}
according to Proposition~\ref{prop1}.
\item Since $r\in p'\vee\Max L(p,q)\wedge q$ implies $p'\leq r$ and hence $\Min U(r,p')=r$ and since $\Max L(p,q)\leq p$ we have
\begin{align*}
(p\rightarrow q)\odot p & =\big(p'\vee\Max L(p,q)\big)\odot p=\bigcup\{r\odot p\mid r\in p'\vee\Max L(p,q)\}= \\
                        & =\bigcup\{\Min U(r,p')\wedge p\mid r\in p'\vee\Max L(p,q)\}= \\
                        & =\bigcup\{r\wedge p\mid r\in p'\vee\Max L(p,q)\}=\big(p'\vee\Max L(p,q)\big)\wedge p=\Max L(p,q)\leq \\
                        & \leq q
\end{align*}
according to Proposition~\ref{prop1}.
\end{enumerate}
\end{proof}

\begin{lemma}
Let $(A,\leq,{}',0,1)$ be a finite orthomodular poset, $x\in(2^A\setminus\{\emptyset\})^T$ and $p\in A^T$. Then
\begin{enumerate}[{\rm(i)}]
\item $x\leq_1p\rightarrow(x\odot p)$ and $x\leq_2p\rightarrow(x\odot p)$,
\item $(p\rightarrow x)\odot p\leq_1x$ and $(p\rightarrow x)\odot p\leq_2x$.
\end{enumerate}
\end{lemma}

\begin{proof}
\
\begin{enumerate}[(i)]
\item We have
\begin{align*}
p\rightarrow(x\odot p) & =p\rightarrow\bigcup\{q\odot p\mid q\in x\}=p\rightarrow\bigcup\{\Min U(q,p')\wedge p\mid q\in x\}= \\
                       & =\bigcup\{p\rightarrow r\mid r\in\bigcup\{\Min U(q,p')\wedge p\mid q\in x\}\}= \\
                       & =\bigcup\{p'\vee\Max L(p,r)\mid r\in\bigcup\{\Min U(q,p')\wedge p\mid q\in x\}\}= \\
                       & =\bigcup\{p'\vee r\mid r\in\bigcup\{\Min U(q,p')\wedge p\mid q\in x\}\}= \\
                       & =p'\vee\bigcup\{\Min U(q,p')\wedge p\mid q\in x\}= \\
											 & =\bigcup\{p'\vee\big(\Min U(q,p')\wedge p\big)\mid q\in x\}=\bigcup\{\Min U(q,p')\mid q\in x\}.
\end{align*}
\item We have
\begin{align*}
(p\rightarrow x)\odot p & =\bigcup\{p\rightarrow q\mid q\in x\}\odot p=\bigcup\{p'\vee\Max L(p,q)\mid q\in x\}\odot p= \\
                        & =\bigcup\{r\odot p\mid r\in\bigcup\{p'\vee\Max L(p,q)\mid q\in x\}\}= \\
                        & =\bigcup\{\Min U(r,p')\wedge p\mid r\in\bigcup\{p'\vee\Max L(p,q)\mid q\in x\}\}= \\
                        & =\bigcup\{r\wedge p\mid r\in\bigcup\{p'\vee\Max L(p,q)\mid q\in x\}\}= \\
                        & =\bigcup\{p'\vee\Max L(p,q)\mid q\in x\}\wedge p= \\
                        & =\bigcup\{\big(p'\vee\Max L(p,q)\big)\wedge p\mid q\in x\}=\bigcup\{\Max L(p,q)\mid q\in x\}.
\end{align*}
\end{enumerate}
\end{proof}

\section{Dynamic pairs}

As mentioned in the introduction, by a {\em dynamic pair} is meant a couple $(P,G)$ of tense operators satisfying axioms (P1) -- (P3). We firstly define tense operators on a finite orthomodular poset $(A,\leq,{}',0,1)$ as follows: Let a time frame $(T,R)$ be given. Define $P,F,H,G\colon A^T\rightarrow(2^A\setminus\{\emptyset\})^T$ by
\begin{align*}
P(q)(s) & :=\Min U(\{q(t)\mid t\mathrel Rs\}), \\
F(q)(s) & :=\Min U(\{q(t)\mid s\mathrel Rt\}), \\
H(q)(s) & :=\Max L(\{q(t)\mid t\mathrel Rs\}), \\
G(q)(s) & :=\Max L(\{q(t)\mid s\mathrel Rt\})
\end{align*}
for all $q\in A^T$ and all $s\in T$.

It is elementary to show that if the orthomodular poset $(A,\leq,{}',0,1)$ is a complete lattice then these tense operators coincide with those mentioned in the introduction. However, in general $P(q)$ for $q\in A^T$ need not be a single element of $A^T$ (i.e.\ a singleton), but may be a non-empty subset of $A^T$. Analogously for the remaining tense operators. Hence, these operators are again ``inexact'' in the sense that they reach maximal (for $H$, $G$) or minimal (for $P$, $F$) incomparable values and we cannot distinguished among them. On the other hand, the results of these operators belong to $A$ contrary to the case when the poset $(A,\leq,{}',0,1)$ is embedded into a complete lattice, the method used in \cite{CP13}.

 Since we want to compose our tense operators we need to solve two essential tasks:
\begin{enumerate}[(i)]
\item Extend $P,F,H,G$ from $A^T$ to $2^{(A^T)}\setminus\{\emptyset\}$ (i.e.\ lifting of the operators).
\item Define how to compose them.
\end{enumerate}
Concerning the first task we define for $B\in2^{(A^T)}\setminus\{\emptyset\}$ and all $s\in T$
\begin{align*}
P(B)(s) & :=\Min U(\{q(t)\mid q\in B\text{ and }t\mathrel Rs\}), \\
F(B)(s) & :=\Min U(\{q(t)\mid q\in B\text{ and }s\mathrel Rt\}), \\
H(B)(s) & :=\Max L(\{q(t)\mid q\in B\text{ and }t\mathrel Rs\}), \\
G(B)(s) & :=\Max L(\{q(t)\mid q\in B\text{ and }s\mathrel Rt\}).
\end{align*}

In order to solve the second task, we should mention that our extended tense operators are not mappings from $2^{(A^T)}\setminus\{\emptyset\}$ to $2^{(A^T)}\setminus\{\emptyset\}$, but from $2^{(A^T)}\setminus\{\emptyset\}$ to $(2^A\setminus\{\emptyset\})^T$. Hence we introduce the so-called {\em transformation function} $\varphi\colon(2^A\setminus\{\emptyset\})^T\rightarrow2^{(A^T)}\setminus\{\emptyset\}$ as follows:
\[
\varphi(x):=\{q\in A^T\mid q(t)\in x(t)\text{ for all }t\in T\}\text{ for all }x\in(2^A\setminus\{\emptyset\})^T. 
\]
By means of the transformation function we can define the {\em composition} $G*P$ of the tense operators $G$ and $P$ by
\[
G*P:=G\circ\varphi\circ P.
\]

\begin{lemma}\label{lem2}
Let $(A,\leq)$ be a poset, $T\neq\emptyset$, $p\in A^T$ and $x,y\in(2^A\setminus\{\emptyset\})^T$. Then
\begin{enumerate}[{\rm(i)}]
\item $\varphi$ is injective,
\item if $z(t):=\{p(t)\}$ for all $t\in T$ then $\varphi(z)=\{p\}$,
\item we have $x\leq y$ if and only if $\varphi(x)\leq\varphi(y)$.
\end{enumerate}
\end{lemma}

\begin{proof}
\
\begin{enumerate}[(i)]
\item Assume $x\neq y$. Then there exists some $t\in T$ with $x(t)\neq y(t)$. Without loss of generality, assume $x(t)\setminus\big(y(t)\big)\neq\emptyset$. Let $a\in x(t)\setminus\big(y(t)\big)$ and $p\in A^T$ with $p(t)=a$ and $p(s)\in x(s)$ for all $s\in T\setminus\{t\}$. Then $p\in\varphi(x)\setminus\big(\varphi(y)\big)$ and hence $\varphi(x)\neq\varphi(y)$.
\item If $z(t):=\{p(t)\}$ for all $t\in T$ then $\varphi(z)=\{q\in A^T\mid q(t)\in z(t)\text{ for all }t\in T\}=\{p\}$.
\item Assume $x\leq y$. Let $p\in\varphi(x)$ and $q\in\varphi(y)$. Then $p(s)\in x(s)$ and $q(s)\in y(s)$ for all $s\in T$. Hence $p(s)\leq q(s)$ for all $s\in T$, i.e.\ $p\leq q$. This shows $\varphi(x)\leq\varphi(y)$. Conversely, assume $\varphi(x)\leq\varphi(y)$. Let $t\in T$, $a\in x(t)$ and $b\in y(t)$. Take $p,q\in A^T$ with $p(t)=a$, $p(s)\in x(s)$ for all $s\in T\setminus\{t\}$, $q(t)=b$ and $q(s)\in y(s)$ for all $s\in T\setminus\{t\}$. Then $p(s)\in x(s)$ and $q(s)\in y(s)$ for all $s\in T$. Hence $p\in\varphi(x)$ and $q\in\varphi(y)$ which implies $p\leq q$ whence $a=p(t)\leq q(t)\leq b$. This shows $x(t)\leq y(t)$. Since $t$ was an arbitrary element of $T$ we conclude $x\leq y$.
\end{enumerate}
\end{proof}

Now we are ready to prove our result. But first we make the following more or less evident observation.

\begin{observation}\label{obs1}
Let $(A,\leq)$ be a finite poset and $B,C$ subsets of $A$ with $C\neq\emptyset$. Then $B\subseteq C$ implies $B\leq_1\Max C$ and $\Min C\leq_2B$.
\end{observation}

\begin{theorem}
Let $\mathbf A=(A,\leq,{}',0,1)$ be a finite orthomodular poset, $(T,R)$ a time frame and $P$ and $G$ tense operators as defined above. Then the couple $(P,G)$ forms a dynamic pair, i.e.
\begin{enumerate}[{\rm(P1)}]
\item $G(1)=1$ and $P(0)=0$,
\item if $p,q\in A^T$ and $p\leq q$ then $G(p)\leq_1G(q)$ and $P(p)\leq_2P(q)$,
\item if $q\in A^T$ then $q\leq_1(G*P)(q)$ and $(P*G)(q)\leq_2q$.
\end{enumerate}
\end{theorem}

\begin{proof}
Let $p,q\in A^T$ and $s\in T$.
\begin{enumerate}[(P1)]
\item $G(1)(s)=\Max L(1)=\Max A=1$ and $P(0)(s)=\Min U(0)=\Min A=0$
\item Assume $p\leq q$. Then
\[
\Max L(\{p(t)\mid s\mathrel Rt\})\subseteq L(\{p(t)\mid s\mathrel Rt\})\subseteq L(\{q(t)\mid s\mathrel Rt\})
\]
and hence
\[
G(p)(s)=\Max L(\{p(t)\mid s\mathrel Rt\})\leq_1\Max L(\{q(t)\mid s\mathrel Rt\})=G(q)(s)
\]
according to Observation~\ref{obs1}. Similarly,
\[
\Min U(\{q(t)\mid t\mathrel Rs\})\subseteq U(\{q(t)\mid t\mathrel Rs\})\subseteq U(\{p(t)\mid t\mathrel Rs\})
\]
and hence
\[
P(p)(s)=\Min U(\{p(t)\mid t\mathrel Rs\})\leq_2\Min U(\{q(t)\mid t\mathrel Rs\})=P(q)(s)
\]
according to Observation~\ref{obs1}.
\item We have
\begin{align*}
              P(q)(s) & =\Min U(\{q(t)\mid t\mathrel Rs\}), \\
\varphi\big(P(q)\big) & =\{r\in A^T\mid r(u)\in P(q)(u)\text{ for all }u\in T\}= \\
                      & =\{r\in A^T\mid r(u)\in \Min U(\{q(t)\mid t\mathrel Ru\})\text{ for all }u\in T\}, \\
          (G*P)(q)(s) & =G\Big(\varphi\big(P(q)\big)\Big)(s)=\Max L(\{r(v)\mid r\in\varphi\big(P(q)\big)\text{ and }s\mathrel Rv\})= \\
                      & =\Max L\big(\bigcup\{\Min U(\{q(t)\mid t\mathrel Rv\})\mid s\mathrel Rv\}\big), \\
                 q(s) & \in L\big(\bigcup\{U(\{q(t)\mid t\mathrel Rv\})\mid s\mathrel Rv\}\big)\subseteq \\
								      & \subseteq L\big(\bigcup\{\Min U(\{q(t)\mid t\mathrel Rv\})\mid s\mathrel Rv\}\big)
\end{align*}
and hence $q(s)\leq_1(G*P)(q)(s)$ according to Observation~\ref{obs1}. Analogously, we have
\begin{align*}
              G(q)(s) & =\Max L(\{q(t)\mid s\mathrel Rt\}), \\
\varphi\big(G(q)\big) & =\{r\in A^T\mid r(u)\in G(q)(u)\text{ for all }u\in T\}= \\
                      & =\{r\in A^T\mid r(u)\in\Max L(\{q(t)\mid u\mathrel Rt\})\text{ for all }u\in T\}, \\
          (P*G)(q)(s) & =P\Big(\varphi\big(G(q)\big)\Big)(s)=\Min U(\{r(v)\mid r\in\varphi\big(G(q)\big)\text{ and }v\mathrel Rs\})= \\
                      & =\Min U\big(\bigcup\{\Max L\big(\{q(t)\mid v\mathrel Rt\})\mid v\mathrel Rs\}\big), \\
                 q(s) & \in U\big(\bigcup\{L(\{q(t)\mid v\mathrel Rt\})\mid v\mathrel Rs\}\big)\subseteq \\
								      & \subseteq U\big(\bigcup\{\Max L(\{q(t)\mid v\mathrel Rt\})\mid v\mathrel Rs\}\big)
\end{align*}
and hence $(P*G)(q)(s)\leq_2q(s)$ according to Observation~\ref{obs1}.
\end{enumerate}
\end{proof}

Similarly one can show that also the couple $(F,H)$ forms a dynamic pair.

\section{Properties of tense operators}

As mentioned in our previous section, our logic based on a finite orthomodular poset is equipped with logical connectives $\odot$ and $\rightarrow$ forming an adjoint pair. The aim of this section is to show that our tense operators satisfy properties considered in classical or non-classical tense logics as collected in \cite{CP15a} although they are defined in an inexact way.

Before formulating of our results, let us illuminate our concepts by the following example.

\begin{example}\label{ex1}
We consider the orthomodular poset $(A,\leq,{}',0,1)$ from Figure~1 and the time frame $(\{1,2,3\},\leq)$. Define $p,q\in A^T$ as follows:
\[
\begin{array}{c|ccc}
  t  & 1  & 2  & 3 \\
\hline
p(t) & i' & i' & f' \\
q(t) & b' & a' & a'
\end{array}
\]
Then
\begin{align*}
                        (p\odot q)(1) & =\Min U(i',b)\wedge b'=i'\wedge b'=d, \\
                        (p\odot q)(2) & =\Min U(i',a)\wedge a'=i'\wedge a'=e, \\
                        (p\odot q)(3) & =\Min U(f',a)\wedge a'=f'\wedge a'=h, \\
			H\big(\varphi(p\odot q)\big)(1) & =\Max L(d)=d, \\
      H\big(\varphi(p\odot q)\big)(2) & =\Max L(d,e)=0, \\
      H\big(\varphi(p\odot q)\big)(3) & =\Max L(d,e,h)=0, \\
                              H(p)(1) & =\Max L(i')=i', \\
                              H(p)(2) & =\Max L(i')=i', \\
                              H(p)(3) & =\Max L(f',i')=\{a,b\}, \\
                              H(q)(1) & =\Max L(b')=b', \\
                              H(q)(2) & =\Max L(a',b')=\{f,i\}, \\
                              H(q)(3) & =\Max L(a',b')=\{f,i\}, \\
          \big(H(p)\odot H(q)\big)(1) & =i'\odot b'=\Min U(i',b)\wedge b'=i'\wedge b'=d, \\
          \big(H(p)\odot H(q)\big)(2) & =i'\odot\{f,i\}= \\
			                                & =\big(\Min U(i',f')\wedge f\big)\cup\big(\Min U(i',i')\wedge i\big)=(1\wedge f)\cup(i'\wedge i)= \\
                                      & =\{0,f\}, \\
          \big(H(p)\odot H(q)\big)(3) & =\{a,b\}\odot\{f,i\}= \\
                                      & =\big(\Min U(a,f')\wedge f\big)\cup\big(\Min U(a,i')\wedge i\big)\cup\big(\Min U(b,f')\wedge f\big)\cup \\
	    						      							& \hspace*{4mm}\cup\big(\Min U(b,i')\wedge i\big)=\{f'\wedge f,i'\wedge i,f'\wedge f,i'\wedge i\}=0, \\
                  (p\rightarrow q)(1) & =i\vee\Max L(i',b')=i\vee d=b', \\
                  (p\rightarrow q)(2) & =i\vee\Max L(i',a')=i\vee e=a', \\
                  (p\rightarrow q)(3) & =f\vee\Max L(f',a')=f\vee h=a', \\
H\big(\varphi(p\rightarrow q)\big)(1) & =\Max L(b')=b', \\
H\big(\varphi(p\rightarrow q)\big)(2) & =\Max L(a',b')=\{f,i\}, \\
H\big(\varphi(p\rightarrow q)\big)(3) & =\Max L(a',b')=\{f,i\}, \\
    \big(H(p)\rightarrow H(q)\big)(1) & =i'\rightarrow b'=i\vee\Max L(i',b')=i\vee d=b', \\
    \big(H(p)\rightarrow H(q)\big)(2) & =i'\rightarrow\{f,i\}=\big(i\vee\Max L(i',f)\big)\cup\big(i\vee\Max L(i',i)\big)= \\
                                      & =\{i\vee0,i\vee0\}=\{i\}, \\
    \big(H(p)\rightarrow H(q)\big)(3) & =\{a,b\}\rightarrow\{f,i\}= \\
                                      & =\big(a'\vee\Max L(a,f)\big)\cup\big(a'\vee\Max L(a,i)\big)\cup\big(b'\vee\Max L(b,f)\big)\cup \\
    																	& \hspace*{4mm}\cup\big(b'\vee\Max L(b,i)\big)=\{a'\vee0,a'\vee0,b'\vee0,b'\vee0\}=\{a',b'\}.
\end{align*}
Moreover,
\begin{align*}
			G\big(\varphi(p\odot q)\big)(1) & =\Max L(d,e,h)=0, \\
      G\big(\varphi(p\odot q)\big)(2) & =\Max L(e,h)=0, \\
      G\big(\varphi(p\odot q)\big)(3) & =\Max L(h)=h, \\
                              G(p)(1) & =\Max L(f',i')=\{a,b\}, \\
                              G(p)(2) & =\Max L(f',i')=\{a,b\}, \\
                              G(p)(3) & =\Max L(f')=f', \\
                              G(q)(1) & =\Max L(a',b')=\{f,i\}, \\
                              G(q)(2) & =\Max L(a')=a', \\
                              G(q)(3) & =\Max L(a')=a', \\
          \big(G(p)\odot G(q)\big)(1) & =\{a,b\}\odot\{f,i\}=0, \\
          \big(G(p)\odot G(q)\big)(2) & =\{a,b\}\odot a'=\big(\Min U(a,a)\wedge a'\big)\cup\big(\Min U(b,a)\wedge a'\big)= \\
					                            & =\{a\wedge a',f'\wedge a',i'\wedge a'\}=\{0,e,h\}, \\
					\big(G(p)\odot G(q)\big)(3) & =f'\odot a'=\Min U(f',a)\wedge a'=f'\wedge a'=h, \\
G\big(\varphi(p\rightarrow q)\big)(1) & =\Max L(a',b')=\{f,i\}, \\
G\big(\varphi(p\rightarrow q)\big)(2) & =\Max L(a')=a', \\
G\big(\varphi(p\rightarrow q)\big)(3) & =\Max L(a')=a', \\
    \big(G(p)\rightarrow G(q)\big)(1) & =\{a,b\}\rightarrow\{f,i\}=\{a',b'\}, \\
    \big(G(p)\rightarrow G(q)\big)(2) & =\{a,b\}\rightarrow a'=\big(a'\vee\Max L(a,a')\big)\cup\big(b'\vee\Max L(b,a')\big)= \\
		                                  & =\{a'\vee0,b'\vee0\}=\{a',b'\}, \\
		\big(G(p)\rightarrow G(q)\big)(3) & =f'\rightarrow a'=f\vee\Max L(f',a')=f\vee h=a'.
\end{align*}
Altogether, we obtain
\[
\begin{array}{c|ccc}
t & 1 & 2 & 3 \\
\hline
   H\big(\varphi(p\odot q)\big)(t)    &     d     &     0     &    0 \\
     \big(H(p)\odot H(q)\big)(t)      &     d     &  \{0,f\}  &    0 \\
H\big(\varphi(p\rightarrow q)\big)(t) &     b'    &  \{f,i\}  & \{f,i\} \\
  \big(H(p)\rightarrow H(q)\big)(t)   &     b'    &     i     &    i \\
   G\big(\varphi(p\odot q)\big)(t)    &     0     &     0     &    h \\
     \big(G(p)\odot G(q)\big)(t)      &     0     & \{0,e,h\} &    h \\
G\big(\varphi(p\rightarrow q)\big)(t) &  \{f,i\}  &     a'    &    a' \\
  \big(G(p)\rightarrow G(q)\big)(t)   & \{a',b'\} & \{a',b'\} &    a'
\end{array}
\]
This shows
\begin{align*}
      H\big(\varphi(p\odot q)\big) & \leq H(p)\odot H(q), \\
H\big(\varphi(p\rightarrow q)\big) & \leq_2H(p)\rightarrow H(q), \\
      G\big(\varphi(p\odot q)\big) & \leq G(p)\odot G(q), \\
G\big(\varphi(p\rightarrow q)\big) & \leq_1G(p)\rightarrow G(q)
\end{align*}
and these inequalities are proper.$\hfill\Box$
\end{example}

We introduce the following notations:

For $p\in A^T$, $x\in(2^A\setminus\{\emptyset\})^T$ and $B\in2^{(A^T)}\setminus\{\emptyset\}$ we define $p'\in A^T$, $x'\in(2^A\setminus\{\emptyset\})^T$ and $B'\in2^{(A^T)}\setminus\{\emptyset\}$ as follows:
\begin{align*}
p'(t) & :=\big(p(t)\big)'\text{ for all }t\in T, \\
x'(t) & :=\{a'\mid a\in x(t)\}\text{ for all }t\in T, \\
   B' & :=\{q'\mid q\in B\}.
\end{align*}
We have $\big(\varphi(x)\big)'=\varphi(x')$ since
\begin{align*}
\big(\varphi(x)\big)' & =\{q'\mid q\in\varphi(x)\}=\{q\in A^T\mid q'\in\varphi(x)\}= \\
                      & =\{q\in A^T\mid\big(q(t)\big)'\in x(t)\text{ for all }t\in T\}=\{q\in A^T\mid q(t)\in x'(t)\text{ for all }t\in T\}= \\
                      & =\varphi(x').
\end{align*}

When combining operators, both tense ones and logical connectives, we must use the transformation function $\varphi$ because of the reasons explained in the previous section. We can state and prove the following propositions.

\begin{proposition}\label{prop3}
Let $(A,\leq,0,1)$ be a finite bounded poset, $(T,R)$ a time frame with reflexive $R$ and $q\in A^T$. Then
\begin{align*}
                  q & \leq(\varphi\circ P)(q), \\
                  q & \leq(\varphi\circ F)(q), \\
(\varphi\circ H)(q) & \leq q, \\
(\varphi\circ G)(q) & \leq q.
\end{align*}
\end{proposition}

\begin{proof}
Let $s\in T$. We have
\begin{align*}
            P(q)(s) & =\Min U(\{q(t)\mid t\mathrel Rs\}), \\
(\varphi\circ P)(q) & =\{p\in A^T\mid p(t)\in\Min U(\{q(u)\mid u\mathrel Rt\})\text{ for all }t\in T\}, \\
                  q & \leq p\text{ for all }p\in(\varphi\circ P)(q), \\
                  q & \leq(\varphi\circ P)(q).
\end{align*}
The result for $F$ is analogous. Now
\begin{align*}
            H(q)(s) & =\Max L(\{q(t)\mid t\mathrel Rs\}), \\
(\varphi\circ H)(q) & =\{p\in A^T\mid p(t)\in\Max L(\{q(u)\mid u\mathrel Rt\})\text{ for all }t\in T\}, \\
                  p & \leq q\text{ for all }p\in(\varphi\circ H)(q), \\
(\varphi\circ H)(q) & \leq q.
\end{align*}
The result for $G$ is analogous.
\end{proof}

\begin{proposition}\label{prop2}
Let $(A,\leq,{}',0,1)$ be a finite orthomodular poset, $(T,R)$ a time frame, $x\in(2^A\setminus\{\emptyset\})^T$, $B,C\in2^{(A^T)}\setminus\{\emptyset\}$ with $B\leq C$ and $s\in T$ and denote by $\varphi$ the transformation function. Then the following hold:
\begin{enumerate}[{\rm(i)}]
\item
\begin{align*}
P\big(\varphi(x)\big)(s) & =\Min U\big(\bigcup\{x(t)\mid t\mathrel Rs\}\big), \\
F\big(\varphi(x)\big)(s) & =\Min U\big(\bigcup\{x(t)\mid s\mathrel Rt\}\big), \\
H\big(\varphi(x)\big)(s) & =\Max L\big(\bigcup\{x(t)\mid t\mathrel Rs\}\big), \\
G\big(\varphi(x)\big)(s) & =\Max L\big(\bigcup\{x(t)\mid s\mathrel Rt\}\big),
\end{align*}
\item $H(B)=P(B')'$ and $G(B)=F(B')'$,
\item $P(B)\leq_2P(C)$, $F(B)\leq_2F(C)$, $H(B)\leq_1H(C)$ and $G(B)\leq_1G(C)$,
\item $H(B)\leq P(B)$ and $G(B)\leq F(B)$.
\end{enumerate}
\end{proposition}

\begin{proof}
\
\begin{enumerate}[(i)]
\item We have
\[
P\big(\varphi(x)\big)(s)=\Min L(\{p(t)\mid p\in\varphi(x)\text{ and }s\mathrel Rt\})=\Min U\big(\bigcup\{x(t)\mid t\mathrel Rs\}\big).
\]
The proof for $F$, $H$ and $G$ is analogous.
\item We have
\begin{align*}
P(B)(s) & =\Min U(\{q(t)\mid q\in B\text{ and }t\mathrel Rs\})=\Min U(\{q'(t)\mid q\in B'\text{ and }t\mathrel Rs\})= \\
        & =\Min U(\{\big(q(t)\big)'\mid q\in B'\text{ and }t\mathrel Rs\})= \\
				& =\big(\Max L(\{q(t)\mid q\in B'\text{ and }t\mathrel RS\}\big)'=H(B')'(s).
\end{align*}
The second assertion can be proved in an analogous way.
\item We have
\begin{align*}
P(C)(s) & =\Min U(\{q(t)\mid q\in C\text{ and }t\mathrel Rs\})\subseteq U(\{q(t)\mid q\in C\text{ and }t\mathrel Rs\})\subseteq \\
        & \subseteq U(\{p(t)\mid p\in B\text{ and }t\mathrel Rs\})
\end{align*}
and hence
\[
P(B)(s)=\Min U(\{p(t)\mid p\in B\text{ and }t\mathrel Rs\})\leq_2 P(C)(s).
\]
Analogously, we obtain $F(B)\leq_2F(C)$. Now
\begin{align*}
H(B)(s) & =\Max L(\{p(t)\mid p\in B\text{ and }t\mathrel Rs\})\subseteq L(\{p(t)\mid p\in B\text{ and }t\mathrel Rs\})\subseteq \\
        & \subseteq L(\{q(t)\mid q\in C\text{ and }t\mathrel Rs\})
\end{align*}
and hence
\[
H(B)(s)\leq_1\Max L(\{q(t)\mid q\in C\text{ and }t\mathrel Rs\})=H(C)(s).
\]
Analogously, we obtain $G(B)\leq_1G(C)$.
\item Since $R$ is serial there exists some $u\in T$ with $u\mathrel Rs$. We have
\begin{align*}
H(B)(s) & =\Max L(\{q(t)\mid q\in B\text{ and }t\mathrel Rs\})\leq\{q(u)\mid q\in B\}\leq \\
        & \leq\Min U(\{q(t)\mid q\in B\text{ and }t\mathrel Rs\})=P(B)(s).
\end{align*}
The second assertion follows analogously.
\end{enumerate}
\end{proof}

By (ii) of Proposition~\ref{prop2} we see that the tense operators $P$ and $F$ are fully determined by means of $H$ and $G$, respectively. Condition (iii) of Proposition~\ref{prop2} shows monotonicity of all tense operators in a finite orthomodular poset.

In the next theorem we verify that the tense operators $P$, $F$, $H$ and $G$ as defined in Section~2 satisfy the composition laws in accordance with known sources, see e.g. \cite{Bu}, \cite{CP15a} or \cite{FGV}.

\begin{theorem}\label{th2}
Let $(A,\leq,0,1)$ be a finite bounded poset and $(T,R)$ a time frame with reflexive $R$. Then
\[
\begin{array}{llll}
P\leq_2P*F, & F\leq_2F*P, & H\leq_1H*P, & G\leq_1G*P, \\
P*H\leq_2P, & F*H\leq_2F, & H\leq_1H*F, & G\leq_1G*F, \\
P*G\leq_2P, & F*G\leq_2F, & H*G\leq_1H, & G*H\leq_1G.
\end{array}
\]
\end{theorem}

\begin{proof}
Let $q\in A^T$. According to Proposition~\ref{prop3} we have $q\leq(\varphi\circ F)(q)$. Hence by (iii) of Proposition~\ref{prop2} we obtain
\[
P(q)\leq_2P\big((\varphi\circ F)(q)\big)=(P\circ\varphi\circ F)(q)=(P*F)(q).
\]
This shows $P\leq_2P*F$. The remaining inequalities can be proved analogously.
\end{proof}

However, in some cases we can prove results which are stronger than those of Theorem~\ref{th2}.

\begin{theorem}\label{th1}
Let $(A,\leq,0,1)$ be a finite bounded poset, $(T,R)$ a time frame with reflexive $R$ and $X\in\{P,F,H,G\}$. Then
\[
H*X,G*X\leq X\leq P*X,F*X,
\]
especially
\[
P\leq P*P, F\leq F*F, H*H\leq H, G*G\leq G.
\]
\end{theorem}

\begin{proof}
If $q\in A^T$ and $s\in T$ then
\[
(\varphi\circ X)(q)=\{r\in A^T\mid r(u)\in X(q)(u)\text{ for all }u\in T\}
\]
and hence
\begin{align*}
(P*X)(q)(s) & =\Min U(\bigcup\{X(q)(t)\mid t\mathrel Rs\})\geq X(q)(s), \\
(F*X)(q)(s) & =\Min U(\bigcup\{X(q)(t)\mid s\mathrel Rt\})\geq X(q)(s), \\
(H*X)(q)(s) & =\Max L(\bigcup\{X(q)(t)\mid t\mathrel Rs\})\leq X(q)(s), \\
(G*X)(q)(s) & =\Max L(\bigcup\{X(q)(t)\mid s\mathrel Rt\})\leq X(q)(s).
\end{align*}
\end{proof}

That $P*P=P$ does not hold in general can be seen from the following example.

\begin{example}
Consider the orthomodular poset $(A,\leq,{}',0,1)$ from Figure~1 and the time frame $(\{1,2,3\},\leq)$ as in Example~\ref{ex1}. Define $r\in A^T$ as follows:
\[
\begin{array}{c|ccc}
  t  & 1 & 2 & 3 \\
\hline
r(t) & a & b & b
\end{array}
\]
Then
\begin{align*}
    P(r)(1) & =\Min U(a)=a, \\
    P(r)(2) & =\Min U(a,b)=\{f',i'\}, \\
    P(r)(3) & =\Min U(a,b)=\{f',i'\}, \\
(P*P)(r)(1) & =\Min U(a)=a, \\
(P*P)(r)(2) & =\Min U(a,f',i')=1, \\
(P*P)(r)(3) & =\Min U(a,f',i')=1.
\end{align*}
and hence
\[
\begin{array}{c|ccc}
    t    & 1 &     2     &     3 \\
\hline
  P(r)   & a & \{f',i'\} & \{f',i'\} \\
(P*P)(r) & a &     1     &     1
\end{array}
\]
$\hfill\Box$
\end{example}

In the following we show that also preserving of the connective $\rightarrow$ with respect to $H$ or $G$ can be derived by the corresponding property for $\odot$.

\begin{lemma}\label{lem1}
Let $(A,\leq,{}',0,1)$ be a finite orthomodular poset, $\odot$ and $\rightarrow$ defined by {\rm(13)} and $B,C,D\in2^A\setminus\{\emptyset\}$. Then $B\odot C\sqsubseteq D$ is equivalent to $B\sqsubseteq C\rightarrow D$.
\end{lemma}

\begin{proof}
First assume $B\odot C\sqsubseteq D$. Then there exist $b\in B$, $c\in C$, $d\in D$ and $a\in b\odot c$ with $a\leq d$. Hence $b\odot c\sqsubseteq d$. By operator left adjointness this implies $b\sqsubseteq c\rightarrow d$, i.e.\ there exists some $e\in c\rightarrow d$ with $b\leq e$. Since $b\in B$ and $e\in C\rightarrow D$ we conclude $B\sqsubseteq C\rightarrow D$. Conversely, assume $B\sqsubseteq C\rightarrow D$. Then there exist $b\in B$, $c\in C$, $d\in D$ and $e\in c\rightarrow d$ with $b\leq e$. Hence $b\sqsubseteq c\rightarrow d$. By operator left adjointness this implies $b\odot c\sqsubseteq d$, i.e.\ there exists some $f\in b\odot c$ with $f\leq d$. Since $f\in B\odot C$ and $d\in D$ we conclude $B\odot C\sqsubseteq D$.
\end{proof}

\begin{theorem}\label{th3}
Let $(A,\leq,{}',0,1)$ be a finite orthomodular poset, $(T,R)$ a time frame and $X,Y,Z\in\{P,F,H,G\}$. Then the following hold:
\begin{enumerate}[{\rm(i)}]
\item Put $i:=1$ if $Z\in\{H,G\}$ and $i:=2$ otherwise and assume
\[
X\big(\varphi(x)\big)\odot Y(q)\leq_iZ\big(\varphi(x\odot q)\big)\text{ for all }x\in(2^A\setminus\{\emptyset\})^T\text{ and all }q\in A^T.
\]
Then
\[
X\big(\varphi(p\rightarrow q)\big)\sqsubseteq Y(p)\rightarrow Z(q)\text{ for all }p,q\in A^T.
\]
\item Put $i:=1$ if $X\in\{H,G\}$ and $i:=2$ otherwise and assume
\[
X\big(\varphi(p\rightarrow x)\big)\leq_iY(p)\rightarrow Z\big(\varphi(x)\big)\text{ for all }p\in A^T\text{ and all }x\in(2^A\setminus\emptyset\})^T.
\]
Then
\[
X(p)\odot Y(q)\sqsubseteq Z\big(\varphi(p\odot q)\big)\text{ for all }p,q\in A^T.
\]
\end{enumerate}
\end{theorem}

\begin{proof}
Let $p,q\in A$ and $x\in2^A\setminus\{\emptyset\}$.
\begin{enumerate}[(i)]
\item Because of divisibility of the operator residuated structure $(A,\leq,\odot,\rightarrow,0,1)$ we have
\[
X\big(\varphi(p\rightarrow q)\big)\odot Y(p)\leq_iZ\Big(\varphi\big((p\rightarrow q)\odot p\big)\Big)=Z\Big(\varphi\big(\Max L(p,q)\big)\Big).
\]
Now $\Max L(p,q)\leq q$ implies $\varphi\big(\Max L(p,q)\big)\leq q$ according to Lemma~\ref{lem2} whence
\[
X\big(\varphi(p\rightarrow q)\big)\odot Y(p)\leq_iZ\Big(\varphi\big(\Max L(p,q)\big)\Big)\leq_iZ(q)
\]
according to (iii) of Proposition~\ref{prop2}. Hence $X\big(\varphi(p\rightarrow q)\big)\odot Y(p)\sqsubseteq Z(q)$. Operator left adjointness (cf.\ \cite{CL}) and Lemma~\ref{lem1} yields
\[
X\big(\varphi(p\rightarrow q)\big)\sqsubseteq Y(p)\rightarrow Z(q).
\]
\item
Because of (i) of Lemma~\ref{lem3} and (iii) of Proposition~\ref{prop2} we obtain
\[
X(p)\leq_iX\Big(\varphi\big(q\rightarrow(p\odot q)\big)\Big)\leq_iY(q)\rightarrow Z\big(\varphi(p\odot q)\big)
\]
and therefore
\[
X(p)\sqsubseteq Y(q)\rightarrow Z\big(\varphi(p\odot q)\big).
\]
Operator left adjointness (cf.\ \cite{CL}) and Lemma~\ref{lem1} yields
\[
X(p)\odot Y(q)\sqsubseteq Z\big(\varphi(p\odot q)\big).
\]
\end{enumerate}
\end{proof}

\begin{remark}
Theorem~\ref{th3} implies that the assertions {\rm(1)} -- {\rm(6)} and {\rm(7)} -- {\rm(12)}, respectively, are in some sense equivalent.
\end{remark}

\section{How to construct a preference relation}

In this section we will present a construction of a binary relation $R$ on $T$ for a given time set $T$ and $\mathbf A$ equipped with the tense operators $P$, $F$, $H$ and $G$.

Let $\mathbf A=(A,\leq,{}',0,1)$ be a finite orthomodular poset, $T$ a time set, $(T,R)$ the time frame. We say that the tense operators $P$, $F$, $H$ and $G$ are {\em constructed by $(T,R)$} if for each $q\in A^T$ and each $s\in T$
\begin{align*}
P(q)(s) & :=\Min U(\{q(t)\mid t\mathrel Rs\}), \\
F(q)(s) & :=\Min U(\{q(t)\mid s\mathrel Rt\}), \\
H(q)(s) & :=\Max L(\{q(t)\mid t\mathrel Rs\}), \\
G(q)(s) & :=\Max L(\{q(t)\mid s\mathrel Rt\}).
\end{align*}
Since $(P,G)$ as well as $(F,H)$ form a dynamic pair, the quintuple $(\mathbf A,P,F,H,G)$ is called a {\em dynamic orthomodular poset}.

Conversely, let $(\mathbf A,P,F,H,G)$ be a dynamic orthomodular poset and $T$ a given time set. The question is how to construct a binary relation $R^*$ on $T$ such that $P$, $F$, $H$ and $G$ are constructed by $(T,R^*)$.

Let $(P,\leq)$ be a poset. On $2^P\setminus\{\emptyset\}$ we define two binary relations $\approx_1$ and $\approx_2$ as follows:
\begin{align*}
A\approx_1B & \text{ if }A\leq_1B\text{ and }B\leq_1A, \\
A\approx_2B & \text{ if }A\leq_2B\text{ and }B\leq_2A
\end{align*}
($A,B\subseteq P$). Since $\leq_1$ and $\leq_2$ are quasiorder relations on $2^P\setminus\{\emptyset\}$, $\approx_1$ and $\approx_2$ are the corresponding equivalence relations on $2^P\setminus\{\emptyset\}$.

Define a binary relation $R^*$ on $T$ as follows:
\begin{enumerate}
\item[(14)] $s\mathrel{R^*}t$ if both $G(q)(s)\leq q(t)\leq F(q)(s)$ and $H(q)(t)\leq q(s)\leq P(q)(t)$ \\
hold for each $q\in A^T$.
\end{enumerate}

\begin{theorem}
Let $(A,\leq,{}',0,1)$ be a finite orthomodular poset, $(T,R)$ a time frame and $P$, $F$, $H$ and $G$ the tense operators constructed by $(T,R)$. Let $R^*$ be the relation defined by {\rm(14)}. Then $R\subseteq R^*$. Let $X\in\{P,F,H,G\}$, $q\in A^T$ and $s\in T$, let $X^*$ denote the tense operator constructed by $(T,R^*)$ corresponding to $X$ and put $i:=1$ if $X\in\{H,G\}$ and $i:=2$ otherwise. Then
\[
X(q)(s)\approx_iX^*(q)(s).
\]
\end{theorem}

\begin{proof}
Let $t\in T$. If $s\mathrel Rt$ then
\begin{align*}
P(q)(t) & =\Min U(\{q(u)\mid u\mathrel Rt\})\geq q(s), \\
F(q)(s) & =\Min U(\{q(u)\mid s\mathrel Ru\})\geq q(t), \\
H(q)(t) & =\Max L(\{q(u)\mid u\mathrel Rt\})\leq q(s), \\
G(q)(s) & =\Max L(\{q(u)\mid s\mathrel Ru\})\leq q(t).
\end{align*}
This shows $R\subseteq R^*$. Now
\begin{align*}
P^*(q)(s) & =\Min U(\{q(t)\mid t\mathrel{R^*}s\})\subseteq U(\{q(t)\mid t\mathrel{R^*}s\})\subseteq U(\{q(t)\mid t\mathrel Rs\}), \\
F^*(q)(s) & =\Min U(\{q(t)\mid s\mathrel{R^*}t\})\subseteq U(\{q(t)\mid s\mathrel{R^*}t\})\subseteq U(\{q(t)\mid s\mathrel Rt\}), \\
H^*(q)(s) & =\Max L(\{q(t)\mid t\mathrel{R^*}s\})\subseteq L(\{q(t)\mid t\mathrel{R^*}s\})\subseteq L(\{q(t)\mid t\mathrel Rs\}), \\
G^*(q)(s) & =\Max L(\{q(t)\mid s\mathrel{R^*}t\})\subseteq L(\{q(t)\mid s\mathrel{R^*}t\})\subseteq L(\{q(t)\mid s\mathrel Rt\})
\end{align*}
and hence
\begin{align*}
  P(q)(s) & =\Min U(\{q(t)\mid t\mathrel Rs\})\leq_2P^*(q)(s), \\
  F(q)(s) & =\Min U(\{q(t)\mid s\mathrel Rt\})\leq_2F^*(q)(s), \\
H^*(q)(s) & \leq_1\Max L(\{q(t)\mid t\mathrel Rs\})=H(q)(s), \\
G^*(q)(s) & \leq_1\Max L(\{q(t)\mid s\mathrel Rt\})=G(q)(s)
\end{align*}
according to Observation~\ref{obs1}. Moreover, we have 
\begin{align*}
   q(t) & \leq P(q)(s)\text{ for all }t\text{ with }t\mathrel{R^*}s, \\
   q(t) & \leq F(q)(s)\text{ for all }t\text{ with }s\mathrel{R^*}t, \\
H(q)(s) & \leq q(t)\text{ for all }t\text{ with }t\mathrel{R^*}s, \\
G(q)(s) & \leq q(t)\text{ for all }t\text{ with }s\mathrel{R^*}t
\end{align*}
and hence
\begin{align*}
P(q)(s) & \in U(\{q(t)\mid t\mathrel{R^*}s\}), \\
F(q)(s) & \in U(\{q(t)\mid s\mathrel{R^*}t\}), \\
H(q)(s) & \in L(\{q(t)\mid t\mathrel{R^*}s\}), \\
G(q)(s) & \in L(\{q(t)\mid s\mathrel{R^*}t\})
\end{align*}
whence
\begin{align*}
P^*(q)(s) & =\Min U(\{q(t)\mid t\mathrel{R^*}s\})\leq_2P(q)(s), \\
F^*(q)(s) & =\Min U(\{q(t)\mid s\mathrel{R^*}t\})\leq_2F(q)(s), \\
H(q)(s) & \leq_1\Max L(\{q(t)\mid t\mathrel{R^*}s\})=H^*(q)(s), \\
G(q)(s) & \leq_1\Max L(\{q(t)\mid s\mathrel{R^*}t\})=G^*(q)(s)
\end{align*}
according to Observation~\ref{obs1}.
\end{proof}

Now we present another construction of a time preference relation on the time set.

Define a binary relation $\bar R$ on $T$ as follows:
\begin{enumerate}
\item[(15)] $s\mathrel{\bar R}t$ if both $(G*X)(q)(s)\leq X(q)(t)\leq(F*X)(q)(s)$ and \\
$(H*X)(q)(t)\leq X(q)(s)\leq(P*X)(q)(t)$ \\
hold for each $X\in\{P,F,H,G\}$ and each $q\in A^T$.
\end{enumerate}

\begin{theorem}
Let $(A,\leq,{}',0,1)$ be a finite orthomodular poset, $(T,R)$ a time frame and $P$, $F$, $H$ and $G$ the tense operators constructed by $(T,R)$. Let $\bar R$ be the relation defined by {\rm(15)}. Then $R\subseteq\bar R$ and for each $X,Y\in\{P,F,H,G\}$, each $q\in A^T$ and each $s\in T$ we have
\[
(Y*X)(q)(s)\approx_i(\bar Y*X)(q)(s)
\]
where $i=1$ if $Y\in\{H,G\}$ and $i=2$ otherwise, and where $\bar Y$ denotes the corresponding tense operator constructed by $(T,\bar R)$.
\end{theorem}

\begin{proof}
Let $X\in\{P,F,H,G\}$, $q\in A^T$ and $s,t\in T$. If $s\mathrel Rt$ then
\begin{align*}
(P*X)(q)(t) & =\Min U(\bigcup\{X(q)(u)\mid u\mathrel Rt\})\geq X(q)(s), \\
(F*X)(q)(s) & =\Min U(\bigcup\{X(q)(u)\mid s\mathrel Ru\})\geq X(q)(t), \\
(H*X)(q)(t) & =\Max L(\bigcup\{X(q)(u)\mid u\mathrel Rt\})\leq X(q)(s), \\
(G*X)(q)(s) & =\Max L(\bigcup\{X(q)(u)\mid s\mathrel Ru\})\leq X(q)(t).
\end{align*}
This shows $R\subseteq\bar R$. Now
\begin{align*}
(\bar P*X)(q)(s) & =\Min U(\bigcup\{X(q)(t)\mid t\mathrel{\bar R}s\})\subseteq U(\bigcup\{X(q)(t)\mid t\mathrel{\bar R}s\})\subseteq \\
                 & \subseteq U(\bigcup\{X(q)(t)\mid t\mathrel Rs\}), \\
(\bar F*X)(q)(s) & =\Min U(\bigcup\{X(q)(t)\mid s\mathrel{\bar R}t\})\subseteq U(\bigcup\{X(q)(t)\mid s\mathrel{\bar R}t\})\subseteq \\
                 & \subseteq U(\bigcup\{X(q)(t)\mid s\mathrel Rt\}), \\
(\bar H*X)(q)(s) & =\Max L(\bigcup\{X(q)(t)\mid t\mathrel{\bar R}s\})\subseteq L(\bigcup\{X(q)(t)\mid t\mathrel{\bar R}s\})\subseteq \\
                 & \subseteq L(\bigcup\{X(q)(t)\mid t\mathrel Rs\}), \\
(\bar G*X)(q)(s) & =\Max L(\bigcup\{X(q)(t)\mid s\mathrel{\bar R}t\})\subseteq U(\bigcup\{X(q)(t)\mid s\mathrel{\bar R}t\})\subseteq \\
                 & \subseteq U(\bigcup\{X(q)(t)\mid s\mathrel Rt\})
\end{align*}
and hence
\begin{align*}
     (P*X)(q)(s) & =\Min U(\bigcup\{X(q)(t)\mid t\mathrel Rs\})\leq_2(\bar P*X)(q)(s), \\
     (F*X)(q)(s) & =\Min U(\bigcup\{X(q)(t)\mid s\mathrel Rt\})\leq_2(\bar F*X)(q)(s), \\
(\bar H*X)(q)(s) & \leq_1\Max L(\bigcup\{X(q)(t)\mid t\mathrel Rs\})=(H*X)(q)(s), \\
(\bar G*X)(q)(s) & \leq_1\Max L(\bigcup\{X(q)(t)\mid s\mathrel Rt\})=(G*X)(q)(s)
\end{align*}
according to Observation~\ref{obs1}. Moreover, we have 
\begin{align*}
    X(q)(t) & \leq(P*X)(q)(s)\text{ for all }t\text{ with }t\mathrel{\bar R}s, \\
    X(q)(t) & \leq(F*X)(q)(s)\text{ for all }t\text{ with }s\mathrel{\bar R}t, \\
(H*X)(q)(s) & \leq X(q)(t)\text{ for all }t\text{ with }t\mathrel{\bar R}s, \\
(G*X)(q)(s) & \leq X(q)(t)\text{ for all }t\text{ with }s\mathrel{\bar R}t
\end{align*}
and hence
\begin{align*}
(P*X)(q)(s) & \in U(\bigcup\{X(q)(t)\mid t\mathrel{\bar R}s\}), \\
(F*X)(q)(s) & \in U(\bigcup\{X(q)(t)\mid s\mathrel{\bar R}t\}), \\
(H*X)(q)(s) & \in L(\bigcup\{X(q)(t)\mid t\mathrel{\bar R}s\}), \\
(G*X)(q)(s) & \in L(\bigcup\{X(q)(t)\mid s\mathrel{\bar R}t\})
\end{align*}
whence
\begin{align*}
(\bar P*X)(q)(s) & =\Min U(\bigcup\{X(q)(t)\mid t\mathrel{\bar R}s\})\leq_2(P*X)(q)(s), \\
(\bar F*X)(q)(s) & =\Min U(\bigcup\{X(q)(t)\mid s\mathrel{\bar R}t\})\leq_2(F*X)(q)(s), \\
(\bar H*X)(q)(s) & \leq_1\Max L(\bigcup\{X(q)(t)\mid t\mathrel{\bar R}s\})=(\bar H*X)(q)(s), \\
(\bar G*X)(q)(s) & \leq_1\Max L(\bigcup\{X(q)(t)\mid s\mathrel{\bar R}t\})=(\bar G*X)(q)(s)
\end{align*}
according to Observation~\ref{obs1}.
\end{proof}

Authors' addresses:

Ivan Chajda \\
Palack\'y University Olomouc \\
Faculty of Science \\
Department of Algebra and Geometry \\
17.\ listopadu 12 \\
771 46 Olomouc \\
Czech Republic \\
ivan.chajda@upol.cz

Helmut L\"anger \\
TU Wien \\
Faculty of Mathematics and Geoinformation \\
Institute of Discrete Mathematics and Geometry \\
Wiedner Hauptstra\ss e 8-10 \\
1040 Vienna \\
Austria, and \\
Palack\'y University Olomouc \\
Faculty of Science \\
Department of Algebra and Geometry \\
17.\ listopadu 12 \\
771 46 Olomouc \\
Czech Republic \\
helmut.laenger@tuwien.ac.at
\end{document}